%% file: arxiv_version.tex
\documentclass[nohyperref]{article}

\usepackage{inputenc}
\usepackage{microtype}
\usepackage{graphicx}
\usepackage{booktabs} 

\usepackage{hyperref}

\usepackage[preprint]{neurips_2021}
\include{preamble}

\usepackage{amsmath}
\usepackage{amssymb}
\usepackage{mathtools}
\usepackage{amsthm}

\usepackage[capitalize,noabbrev]{cleveref}

\theoremstyle{plain}
\newtheorem{theorem}{Theorem}[section]

\newtheorem{lemma}[theorem]{Lemma}

\theoremstyle{definition}

\theoremstyle{remark}

\title{Simplifying deflation for non-convex optimization with applications in Bayesian inference and topology optimization}

\author{%
  Mohamed~Tarek \\
  Pumas-AI Inc., U.S. \\
  The University of Sydney Business School, Australia \\
  \texttt{mohamed82008@gmail.com}
  \And
  Yijiang~Huang \\
  Department of Architecture, Massachusetts Institute of Technology, U.S. \\
  \texttt{yijiangh@mit.edu}
}

\begin{document}
\maketitle

\begin{abstract}
Non-convex optimization problems have multiple local optimal solutions.
Non-convex optimization problems are commonly found in numerous applications.
One of the methods recently proposed to efficiently explore multiple local optimal solutions without random re-initialization relies on the concept of deflation.
In this paper, different ways to use deflation in non-convex optimization and nonlinear system solving are discussed.
A simple, general and novel deflation constraint is proposed to enable the use of deflation together with existing nonlinear programming solvers or nonlinear system solvers.
The connection between the proposed deflation constraint and a minimum distance constraint is presented.
Additionally, a number of variations of deflation constraints and their limitations are discussed.
Finally, a number of applications of the proposed methodology in the fields of approximate Bayesian inference and topology optimization are presented.
\end{abstract}

\section{Introduction}

Non-convex optimization problems are used in numerous applications including: machine learning, mechanical design, economics, and the design of clinical trials, among many other applications.
One of the fundamental challenges of non-convex optimization is the existence of multiple local minimizers. 
Local first and second order optimization algorithms often get stuck in a particular local minimizer without any guarantees that this is the best solution that can be found. 
The ability to explore multiple local minimizers is often important in practice to find better solutions or to provide more diverse choices to decision makers if all the choices are nearly equally good or the optimization objectives cannot capture users' preferences completely.
For example, multiple car body designs can be reported and the decision-makers can choose one based on the subjective aesthetic appeal.

Despite the popularity of using standard non-convex optimization techniques with various random or heuristic restart strategies to find multiple optima \citep{rinnooy_kan_stochastic_1987,kaelo_variants_2006,arnoud_benchmarking_2019}, these methods can be computationally inefficient because they don't protect against converging to the same solution from different starting points.
In contrast, deflation-based methods \citep{papadopoulos_computing_2021} have the benefit of provably converging to distinct solutions under certain assumptions, even when starting from {\it the same initial guess}.
However, originally developed for solving nonlinear equation system \citep{brow1971deflation,farrell_computation_2016}, adapting deflation to non-convex optimization problems requires significant adaptation of the non-convex optimization algorithm \citep{papadopoulos_computing_2021}, which prevents it from being used with a broader class of problems, algorithms and applications.

In this paper, we propose a simple, yet effective and provably correct way of making use of deflation by reformulating the non-convex optimization problem, instead of adapting the backend optimization algorithm.
We then prove the equivalence of this approach to a minimum distance constraint under certain assumptions.
The proposed approach has the benefit of (1) being simpler to implement since it's a formulation change rather than an algorithmic change; (2) being more flexible allowing the use of arbitrary optimization algorithms that are suitable for the problem class at hand.
We show a number of examples from different applications, each using deflation together with the most suitable and/or popular optimization algorithm for the respective application, which is not achievable by previous deflation-based approaches.

One limitation of the proposed approach is that it requires the handling of a non-convex, inequality constraint even if the original problem was unconstrained. However, one way to workaround this limitation is demonstrated in the examples section.

\section{Related work}

\textbf{Systematic multi-start.}
Systematically restarting the optimization multiple times from random or deterministically diverse \citep{latin_hypercube,kucherenko_application_2005} initial solutions is a common strategy to find multiple local minimizers in non-convex optimization. Some popular algorithms following this approach are the multi-level single-linkage (MLSL) algorithm \citep{rinnooy_kan_stochastic_1987}, the controlled random search (CRS) \citep{kaelo_variants_2006} algorithm, and the TikTak algorithm \citep{arnoud_benchmarking_2019}. However, all the approaches above can be computationally expensive and wasteful since restarting the optimization algorithm from a different initial solution may not give a different local minimizer.

\textbf{Divide and conquer} Other similar algorithms rely on sub-dividing the search space into smaller hyper-rectangles to narrow down the search space for each sub-problem. Algorithms in this category include the StoGo algorithm \citep{stogo1,stogo2} and the DIviding RECTangles (DIRECT) algorithm \citep{jones_lipschitzian_1993,gablonsky_locally-biased}. However, these approaches don't scale well with the number of decision variables $m$ since the number of fixed size hyper-rectangles one can divide an $m$ dimensional solution space into grows exponentially with $m$.

\textbf{Hyperparameter and bilevel optimization.} Alternatively, techniques from hyper-parameter optimization \citep{li2018hyperband,falkner2018bohb} or bilevel optimization \citep{bilevel2018} can be used to optimize the starting point of the lower level optimization algorithm.
One can compose a global search algorithm such as an evolutionary algorithm \citep{GendPotv10} and a local search algorithm together to create an algorithm that can explore different initial solutions and find best local minimizers in different neighbourhoods.
This family of global-local algorithms is also sometimes termed memetic algorithms \citep{Moscato2010}.

\textbf{Global metaheuristics.} Beside their use in memetic algorithms, global metaheuristic optimization algorithms \citep{GendPotv10} can also be used as standalone algorithms but these algorithms don't tend to scale well to large problems since they usually don't exploit the often available gradients and sometimes Hessians of objective and constraint functions.

\textbf{Tunneling-based multi-start.} Tunneling \citep{tunneling1,tunneling2,tunneling3,tunneling4,zhang_finding_2018} is another heuristic technique often used to find multiple local minimizers by finding a sufficiently different starting point with a similar or better objective value as the best solution found so far. 
The optimization problem is then re-solved starting from this new point hoping to converge to a different solution.
However, this is a 2 step approach consisting of first solving the optimization problem and then finding a new, sufficiently different initial point. And the success to find distinct minimizers depends on the success to find a sufficiently different starting point.

\textbf{Exact global optimization.} 
Beside the use of restarts, there are also some well known exact global optimization algorithms for some classes of non-convex optimization \citep{tawarmalani_polyhedral_2005,sahinidis:baron:21.1.13,BeLeLiMaWa08,SCIP,EAGO,alpine_JOGO2019,alpine_CP2016,ratschek,gecode}. 
These approaches however tend to require the explicit analytical mathematical expressions of the objective and constraint functions so they are not suitable for black-box non-convex optimization, and they typically don't scale well to large problems in practice, where the exactness guarantee has an exponential computational time complexity in the number of variables.



\textbf{Deflation-based multi-start.}
Deflation is a recently proposed technique that employs Newton-like methods to find multiple solutions of a nonlinear systems of equations~\citep{brow1971deflation, farrell_computation_2016, farrell2020deflation}.
Assuming the underlying Newton-like algorithm converges, under certain assumptions, a deflation-based solver is guaranteed to converge to a unique solution each time the solver is started from the same initial solution. 
This has the promise of being much less wasteful than multi-start optimization approaches. 
However, one problem with deflation is that often the assumptions under which convergence to a different solution is guaranteed are not easy to verify in practice. 
That said, the method was shown to perform well in practical applications.

Deflation was also used in a primal-dual interior point optimization algorithm \citep{primaldual} to find multiple locally optimal designs in mechanical design optimization problems \citep{papadopoulos_computing_2021}.
However, this deflated optimizer had to re-invent a primal-dual interior point optimization algorithm \citep{primaldual} since off-the-shelf solvers such as the Interior Point OPTimizer (IPOPT) \citep{wachter2006implementation} could not be used directly for reasons to be presented in this work. 
Not being able to use existing optimization solvers is a huge limitation of this approach since different optimization algorithms tend to be more suitable for different problems and applications.
%
%
\section{Background}

\subsection{Sufficient optimality conditions for regular points} \label{sec:opt_cond}

We aim to find multiple solutions for the following nonlinear program (NLP):
\begin{mini}|l|[3]
  {\bm{x}\in\RR^n}{f(\bm{x})}{}{}
  \addConstraint {\bm{c}(\bm{x})}{= \bm{0}}{\label{form:eq}}
  \addConstraint \bm{l} \leq {\bm{x} \leq \bm{u}}{}{}
\end{mini}
where $\bm{l} \in (-\infty,\infty)^n, \bm{u} \in (-\infty,\infty)^n$ are the finite (for simplicity) lower and upper bounds of the variable $\bx$.
The objective function $f: \RR^n \xrightarrow{} \RR$ and the equality constraints $c: \RR^n \xrightarrow{} \RR^m$, with $m \leq n$ are assumed to be twice continuously differentiable.
Problems with general nonlinear inequality constraints $\bm{d}(\bx) \leq \bm{0}$ can be converted to equality constraints by adding slack variables.

Let:
\begin{align}
  \mathcal{L}(\bm{x}, \bm{\lambda}, \bm{z}_+, \bm{z}_-) = f(\bm{x}) + \bm{c}(\bm{x})^T \bm{\lambda} + \nonumber\\
  (\bm{x} - \bm{u})^T \bm{z}_+  - (\bm{x} - \bm{l})^T \bm{z}_-
\end{align}

where $\bm{\lambda} \in \mathbb{R}^m$ is the Lagrangian multiplier vector of the equality constraints, $\bm{z}_- \in \mathbb{R}^n_+, \bm{z}_+ \in \mathbb{R}^n_+$ are the Lagrangian multiplier vectors of the bound constraints. 

If $\bm{x}$ is regular and is a local minimizer of the NLP, then $\exists (\bm{\lambda}, \bm{z}_+, \bm{z}_-)$ such that:
\begin{subequations}
\begin{align}
  \nabla_{\bm{x}} \mathcal{L}(\bm{x}, \bm{\lambda}, \bm{z}_+, \bm{z}_-) = \bm{0} \label{eqn:stationarity} \\
  \bm{c}(\bm{x}) = \bm{0}  \label{eqn:primal_feasible_1} \\
  \bm{l} \leq \bm{x} \leq \bm{u}  \label{eqn:primal_feasible_2}  \\
  \bm{z}_+ \geq \bm{0}  \label{eqn:dual_feasible_1}  \\
  \bm{z}_- \geq \bm{0} \label{eqn:dual_feasible_2} \\
  (\bm{x} - \bm{u})^T \bm{z}_+ = 0 \label{eqn:complementarity_1} \\
  (\bm{x} - \bm{l})^T \bm{z}_- = 0 \label{eqn:complementarity_2}
\end{align}
\label{eq:nlp_kkt}
\end{subequations}
and
\begin{align}
  \nabla \bm{c}(\bm{x})^T \nabla_{\bm{x}\bm{x}}^2 \mathcal{L}(\bm{x}, \bm{\lambda}, \bm{z}_+, \bm{z}_-) \nabla \bm{c}(\bm{x}) \succcurlyeq \bm{0} \label{eq:second_order}
\end{align}

\cref{eq:nlp_kkt,eq:second_order} are known as the first and second order Karush-Kuhn-Tucker (KKT) sufficient conditions for optimality respectively. A point $\bm{x}$ that satisfies \cref{eq:nlp_kkt} is typically called a KKT point \citep{nocedal2006numerical}.
Finding multiple solutions of the NLP can be reduced to finding solutions that satisfy \cref{eq:nlp_kkt}, while indirectly enforcing \cref{eq:second_order} via following descent directions.

\subsection{Deflation} \label{sec:deflation}

Deflation is a technique that systematically modifies a nonlinear problem to guarantee that Newton’s method will not converge to a known root, thus enabling unknown roots to be discovered from the same initial guess~\citep{farrell_computation_2016}.

\paragraph{Deflation for nonlinear equation solving}
\citet{farrell_computation_2016} proved that under certain conditions, solving a system of $n$ nonlinear equations:
\begin{align*}
    \bm{F}(\bm{x}) = \bm{0}
\end{align*}
starting from the same initial solution $\bm{x}_0$ can converge to multiple locally optimal solutions by applying a deflation operator whenever a solution is found. Let the first solution found be $\bm{x}_1$. The deflation operator can then be defined as:
\begin{subequations}
\begin{align*}
    \bm{M}(\bm{x}; \bm{x}_1) = m(\bm{x}; \bm{x}_1) \mathcal{I} \\
    m(\bm{x}; \bm{x}_1) = ||\bm{x} - \bm{x}_1||^{-p} + \sigma
\end{align*}
\end{subequations}
for some power $p$ and shift $\sigma$, where $\mathcal{I}$ is the $n \times n$ identity matrix. The deflated nonlinear system of equations is given by:
\begin{align*}
    \bm{M}(\bm{x}; \bm{x}_1) \bm{F}(\bm{x}) = \bm{0}
\end{align*}
After solving the system once, one can solve the deflated system and obtain a new solution $\bm{x}_2$. One of the assumptions required to prove the convergence of the proposed algorithm to a different solution $\bm{x}_2$ is:
\begin{align} \label{eqn:assumption}
    \lim_{\bm{x} \to \bm{x}_1} ||\bm{M}(\bm{x}; \bm{x}_1) \bm{F}(\bm{x})|| > 0
\end{align}
To deflate away from multiple found solutions, the original method proposes multiplying the operators and solve $\prod_{k=1}^{\tilde{K}} \bm{M}(\bx; \bx_k) \, \bm{F}(\bx) = 0$, where $\tilde{K}$ is the number of found solutions.
To improve numerical stability, summation instead of multiplication can also be used.

\paragraph{Deflation for NLP optimization}

To apply deflation to finding multiple solutions of NLP, one can write most of \cref{eq:nlp_kkt} as $\bm{F}(\bm{x}, \bm{\lambda}) = 0$ (see the appendix for more details) and solve for different solutions of the deflated system:
$$\bm{M}(\bx; \bx_1) \bm{F}(\bm{x}, \bm{\lambda}) = 0$$
while ensuring a descent direction is taken at every step.
However, solving this nonlinear system of equations using Newton-like algorithms requires evaluating the Jacobian of the residual of the deflated nonlinear system of equations.
One of the most popular existing implementations of the primal-dual interior point algorithm in IPOPT assumes the Jacobian of the top block of $\bm{F}$ to be symmetric (since it's the Hessian of the Lagrangian).
This makes it difficult to reuse IPOPT with deflation directly.
Arguably, this is also the main motivation for \citet{papadopoulos_computing_2021} developing a specialized primal-dual interior point algorithm from scratch for use with deflation.
For more details on the derivations relevant to this discussion see \cref{sec:appendix_kkt_logbarrier}.




\section{Method} \label{sec:generalizations}

In this section, the use of deflation and the choice of the optimization algorithm to use will be decoupled by making use of a deflation constraint that only requires a change in the formulation of the optimization problem without specifying the algorithm.

\subsection{Formulating deflation as a constraint}

Instead of deflating the optimality conditions, one can create a new constraint and variable that enforce the same deflation effect. Let $y \geq 0$ be a new variable. One possible constraint to add is:
\begin{align}
    m(\bm{x}; \bm{x}_1) = ||\bm{x} - \bm{x}_1||^{-p} + \sigma \leq y
\end{align}
If after adding this constraint, the optimal solution $\bm{x}_2$ obtained has a finite $y$ in exact arithmetic, then $\bm{x}_1 \neq \bm{x}_2$ since $\lim_{\bm{x} \to \bm{x}_1} m(\bm{x}; \bm{x}_1) = \infty$. One can further show that every KKT point of the new NLP with a finite $y$ is a KKT point of the original formulation.

\begin{lemma} \label{lemma:1}
If $(\bm{x}^*, y^*)$ is a regular KKT point to the following NLP:
\begin{mini}|l|[3]
  {\bm{x}, y}{f(\bm{x})}{}{\label{form:deflated}}
  \addConstraint {\bm{c}(\bm{x})}{= \bm{0}}{}
  \addConstraint {\bm{m}(\bm{x}; \bm{x}_1)}{\leq y}{}
  \addConstraint \bm{l} \leq {\bm{x} \leq \bm{u}}{}{}
\end{mini}
    for a finite $y^*$, and $\bm{m}$ is bounded from below, then $\bm{x}^*$ is a regular KKT point to problem \ref{form:eq} and $\bm{x}^* \neq \bm{x}_1$.
\end{lemma}

\begin{proof}
Let the Lagrangian multipliers of the equality constraints be $\bm{\lambda}$, $\bm{z}_-$ be those of the $\geq$ bound constraints, and $\bm{z}_+$ be those of the $\leq$ bound constraints. Let the additional Lagrangian multiplier of the deflation constraint be $\eta$. The stationary conditions of problem \ref{form:deflated} are:
\begin{subequations}
\begin{align*}
  \nabla_{\bm{x}} f(\bm{x}) + \nabla_{\bm{x}} \bm{c}(\bm{x})^T \bm{\lambda} + \bm{z}_+ + \bm{z}_- + \eta \nabla_{\bm{x}} m(\bm{x}; \bm{x}_1) = \bm{0} \\
  \eta = 0
\end{align*}
\end{subequations}

Since $\eta$ will be 0 at any KKT point, the stationarity conditions of problem \ref{form:eq} will be satisfied.
\begin{align}
    \nabla_{\bm{x}} f(\bm{x}) + \nabla_{\bm{x}} \bm{c}(\bm{x})^T \bm{\lambda} + \bm{z}_+ - \bm{z}_- = \bm{0}
\end{align}
Additionally, the complementarity condition of the deflation constraint and the dual feasibility constraint $\eta \geq 0$ are trivially satisfied at $\eta = 0$. Since the constraints of problem \ref{form:eq} are a subset of the constraints of problem \ref{form:deflated}, $\bm{x}^*$ must be feasible to the original problem and the complementarity conditions of those constraints must be satisfied.

Given that $y^*$ is finite, $(\bm{x}^*, y^*)$ is feasible to the deflation constraint in problem \ref{form:deflated} and $m$ is bounded from below, then $m(\bm{x}^*; \bm{x}_1)$ must be finite which implies that $\bm{x}^* \neq \bm{x}_1$. This completes the proof.
\end{proof}

Note that the proof above can be trivially generalized to inequality constrained and even conic constrained nonlinear programs. Therefore, this deflation constraint approach is a completely generic and non-invasive way to use deflation in optimization.

Sine the deflation constraint approach only requires a change in the formulation optimized rather than the optimization routine, any KKT seeking nonlinear programming algorithm, e.g. the method of moving asymptotes \citep{Svanberg1987,Svanberg2002} or the augmented Lagrangian algorithm \citep{Bertsekas1996} can be used to solve the deflated formulation. This is a much more generic and simpler way to use deflation than deflating all of the optimality conditions.

\subsection{Alternative deflation constraint}

One can also avoid the introduction of an additional variable $y$ replacing it with a large finite constant $M$.

\begin{lemma}
If $(\bm{x}^*)$ is a regular KKT point to the following NLP:
\begin{mini}|l|[3]
  {\bm{x}}{f(\bm{x})}{}{\label{form:deflated2}}
  \addConstraint {\bm{c}(\bm{x})}{= \bm{0}}{}
  \addConstraint {\bm{m}(\bm{x}; \bm{x}_1)}{\leq M}{}
  \addConstraint \bm{l} \leq {\bm{x} \leq \bm{u}}{}{}
\end{mini}
for some finite constant $M$, the constraint $\bm{m}(\bm{x}; \bm{x}_1) \leq M$ is satisfied at a strict inequality, and $\bm{m}$ is bounded from below, then $\bm{x}^*$ is a regular KKT point to problem \ref{form:eq} and $\bm{x}^* \neq \bm{x}_1$.
\end{lemma}

\begin{proof}
Let the Lagrangian multipliers of the equality constraints be $\bm{\lambda}$, $\bm{z}_-$ be those of the $\geq$ bound constraints, and $\bm{z}_+$ be those of the $\leq$ bound constraints. Let the additional Lagrangian multiplier of the deflation constraint be $\eta$. The stationary conditions of problem \ref{form:deflated2} are:
\begin{subequations}
\begin{align*}
  \nabla_{\bm{x}} f(\bm{x}) + \nabla_{\bm{x}} \bm{c}(\bm{x})^T \bm{\lambda} + \bm{z}_+ - \bm{z}_- + \eta \nabla_{\bm{x}} m(\bm{x}; \bm{x}_1) = \bm{0}
\end{align*}
\end{subequations}

By the complimentarity slackness conditions, since the deflation constraint is satisfied at a strict inequality, $\eta$ must be equal to $0$ at any KKT point. The rest of the proof is identical to the proof of Lemma \ref{lemma:1}
\end{proof}

\subsection{Simple deflation for nonlinear systems} \label{sec:nonlinear_systems}

Much like in optimization, the following deflation equality constraint can be added to a nonlinear system of equations:
\begin{align}
    m(\bm{x}; \bm{x}_1) = y
\end{align}
solving for both $\bm{x}$ and $y$. It is trivial to see that if the algorithm converges to a solution $(\bm{x}^*, y^*)$ with a finite $y^*$, then $\bm{x}^*$ will satisfy the original set of equations and $\bm{x}^* \neq \bm{x}_1$.

\subsection{Deflating multiple intermediate solutions} \label{sec:multiple_points}

In \citet{papadopoulos_computing_2021}, the deflation operator was used to deflate away from the barrier sub-problems' intermediate solutions. 
This is a nice way to automatically adapt the deflation power by ensuring that we deflate away from the entire critical path of the interior point optimizer next time the system is solved.
It can also be used to provide accelerated convergence by encouraging solutions to intermediate barrier sub-problems to be more different from the previous barrier sub-problems. 
The same approach can be used in any nonlinear programming algorithm by using a manual callback and adding new "known solutions" to the deflation operator every fixed number of iterations. 

\subsection{Deflation constraint is a minimum distance constraint} \label{sec:distance_constraint}

The deflation constraint is equivalent to the following distance constraint assuming finite non-zero $m > 0$ and $y > \sigma$:
\begin{align}
    ||\bm{x} - \bm{x}_1||^{p} \geq z
\end{align}
where $z = \frac{1}{y - \sigma}$. In this case, it is clear what the deflation algorithm does. It just puts a constraint on the distance to known solutions. If the NLP optimizer approaches $z = 0$ or $y = \infty$, then the deflation operator may not be swaying the solution enough away from known solutions.
When using the fixed, large $M$ formulation from equation \ref{form:deflated2}, instead of $y$ tending to infinity, the deflation constraint will be satisfied at equality when deflation fails to make the algorithm converge to a different solution.
The main hyperparameters that can be tuned in the deflation constraint are the distance measure used and $p$ to ensure convergence to a different solution.

\subsection{Different distance measures} \label{sec:distance_measures}

One can change the distance measure used in the deflation function to more interesting choices than a simple power of a 2-norm. The following are different ways to change the distance measure to achieve different desired effect:
\begin{enumerate}
    \item Use a $\ell_q$ norm distance measure instead of a power $p$ of a $\ell_2$ norm, or use a power $p$ of a $\ell_q$ norm distance measure.
    \item Use a positive semi-definite weight matrix $\bm{Q}$ and define the distance as the generalized $\ell_2$ norm distance: $(\bm{x} - \bm{x}_1)^T \bm{Q} (\bm{x} - \bm{x}_1)$. If $\bm{Q}$ is diagonal, it can be used to give different weights to different deviation terms.
    \item The distance measure can detect symmetries in the solution space if 2 solutions are numerically different but practically identical. In topology optimization for example, it may make more sense to define the distance measure in the pseudo-densities $\bm{\rho}$ space after applying density filtering, interpolation and Heaviside projection to the solution $\bm{x}$, rather than calculating distance in the $\bm{x}$ space. This more accurately reflects the desire of the optimizer to obtain different designs rather than different, symmetric representations of the same design.
    \item In topology optimization also, the number of design variables can be extremely large with many similar looking designs that have slightly different numerical values. A particular modification to the distance measure can therefore be useful to workaround this curse of dimensionality problem. Let $\sigma = 0$ and consider the following deflation function:
    \begin{align}
        m(\bm{x}, \bm{x}_1) = \max(||\bm{x} - \bm{x}_1|| - r, 0)^{-p}
    \end{align}
    This deflation function treats all solutions in the hyper-ball of radius $r$ around $\bm{x}_1$ as identical to $\bm{x}_1$, i.e. the distance between them is 0. This function is smooth and twice differentiable for all $\bm{x} \in \{\bm{x}: ||\bm{x} - \bm{x}_1|| > r \}$. A KKT solution with a finite $y$ (in exact arithmetic) can be similarly shown to be equivalent to finding a new optimal solution outside the hyper-ball. This new parameter $r$ can be useful for high dimensional problems in topology optimization where 2 designs can be numerically different but visually and practically identical. Another side advantage of using $r > 0$ is that even if deflation fails to push the optimization algorithm away from $\bm{x}_1$, the new solution obtained is guaranteed to be outside the hyper-ball.
    \item In variational inference, the Kullback-Leibler (K-L) divergence can be used to deflate away from known distributions that locally optimize the K-L divergence to the posterior distribution.
    \item In deep learning, a distance measure between the neuron values can be used instead of the weights and biases to account for symmetries in the weights.
\end{enumerate}

\subsection{Potential problems with deflation} \label{sec:deflation_problems}

While deflation has proven rather successful in a number of root finding and optimization applications, it is not free of limitations which need to be addressed or acknowledged when implementing or using the algorithm. The following are some common problems associated with the proposed deflation constraint most of which are also concerns when using the traditional deflation method.
\begin{enumerate}
    \item The deflation effect may not be strong enough where the algorithm can still approach $\bm{x}_1$, asymptotically increasing $y$ to $\infty$. This is similar to what can happen in the classic deflation when assumption \ref{eqn:assumption} is violated.
    \item The optimization is not done in exact arithmetic but rather up to machine precision. This can make the algorithm converge to a solution with a finite but large value value of $y^*$ instead of overflowing to $\infty$. Therefore, the finiteness of $y^*$ may not be a good enough indication that the algorithm has converged to a solution other than $\bm{x}_1$.
    \item Deflation guarantees that if convergence happens, the solution will be different. However, it doesn't guarantee that convergence will happen to begin with. Even more so, deflating away from 1 solution is likely to push the optimizer away from other nearby locally optimal solutions if they exist. This can make it particularly challenging to fine-tune the hyperparameters of the algorithm. However arguably this can also be a desirable effect of deflation since it means that more diverse solutions are naturally more likely to be output by the algorithm.
    \item Careful selection of the power $p$ and other hyperparameters used in the chosen distance measure is required.
\end{enumerate}

\section{Examples}

In this section, a number of applications of deflation and non-convex optimization are showcased from machine learning and topology optimization. The main highlight of this section is that known popular algorithms were used or minimally modified to solve the deflation sub-problems.
In all of the examples, the same initial solution was used in the deflation sub-problems to showcase the efficacy of the approach proposed. 
Implementations, including detailed hyperparameter settings, can be found in the supplementary code\footnote{\url{https://github.com/JuliaTopOpt/deflation_examples/}}.

\subsection{Classical variational inference}\label{sec:examples_var_inf}

In this section, an example of the use of deflation in variational inference will be demonstrated. 
For simplicity, the log joint probability in variational inference is replaced by the log probability of a 1D mixture of $10$ Gaussian distributions. 
The variational family used is a 1D Gaussian. 
The automatic differentiation variational inference (ADVI) \citep{advi2015} algorithm was used together with the decayed ADAGrad \citep{adagrad2011} stochastic optimization algorithm from the AdvancedVI.jl \footnote{\url{https://github.com/TuringLang/AdvancedVI.jl}} package 
to minimize a stochastic estimator of the negative of the evidence lower bound (ELBO). 
Since there is no data, the negative of the ELBO is equal to the K-L divergence. Distributions.jl \footnote{\url{https://github.com/JuliaStats/Distributions.jl}} was used to define the mixture of Gaussians, the variational family, and the estimator of the K-L divergence objective function. 
The deflated variational approximation problem can be formulated using the loss function $L(\bm{\theta})$ as follows:
\begin{mini}
  {\bm{\theta}, y}{L(\bm{\theta})}{}{\label{form:vi}}
  \addConstraint \sum_{k=1}^{\tilde{K}} {\bm{m}(\bm{\theta}; \bm{\theta}_k)}{\leq y}{}
\end{mini}
where $\bm{\theta}$ is the vector of parameters of the variational family and $\bm{\theta}_k$ is the previously found local minimizer of the loss function from subproblem $k$, and $\tilde{K}$ is the total number of found solutions so far.
$\bm{m}$ and $y$ are the deflation function and variable respectively as described in the previous sections.

Since the objective is a stochastic estimator of the K-L divergence and the constraint is non-convex, a log-barrier approach \citep{logbarrier} is used to transform the stochastic constrained problem to an unconstrained one. The log-barrier formulation becomes:
\begin{mini}|l|[3]
  {\bm{\theta}, y}{L(\bm{\theta}) - r \times \log \bigg{(y - \sum_{k=1}^{\tilde{K}} \bm{m}(\bm{\theta}; \bm{\theta}_k)\bigg)}}{}{\label{form:vi-log}}
\end{mini}
where $r$ is defined as a decaying coefficient approaching 0. The above problem was solved 10 times from the same initial solution $\bm{\theta}_0$ each time converging to different Gaussian approximations and appending it to the list of found solutions. The distance function used between solutions in the deflation function was the analytic K-L divergence between the 2 Gaussian distributions offset by a radius $r$. Let $d(\bm{\theta}) = \mathcal{N}(\mu = \bm{\theta}[1], \sigma = \exp(\bm{\theta}[2]))$ be the Gaussian distribution obtained from solution $\bm{\theta}$. The K-L divergence based deflation function $m$ used was therefore:
\begin{align}
    m(\bm{\theta}; \bm{\theta}_k) = \max \Bigg( \text{div} \Big( d(\bm{\theta}), d(\bm{\theta}_k) \Big) - 1, 0 \Bigg)^{-3}
\end{align}
where $\text{div}$ is the K-L divergence between 2 Gaussian distributions. Note that the minimum radius of 1 was used to enforce convergence to a different distribution even if the convergence assumptions of deflation were violated. The optimal Gaussian distributions' obtained and the target distribution's probability density functions are shown in figure \ref{fig:vi}. The results indicate that deflation was successful at generating reasonable mode-seeking approximations of the target mixture of Gaussians.

\begin{figure}[htb]
     \centering
     \includegraphics[width=0.9\columnwidth]{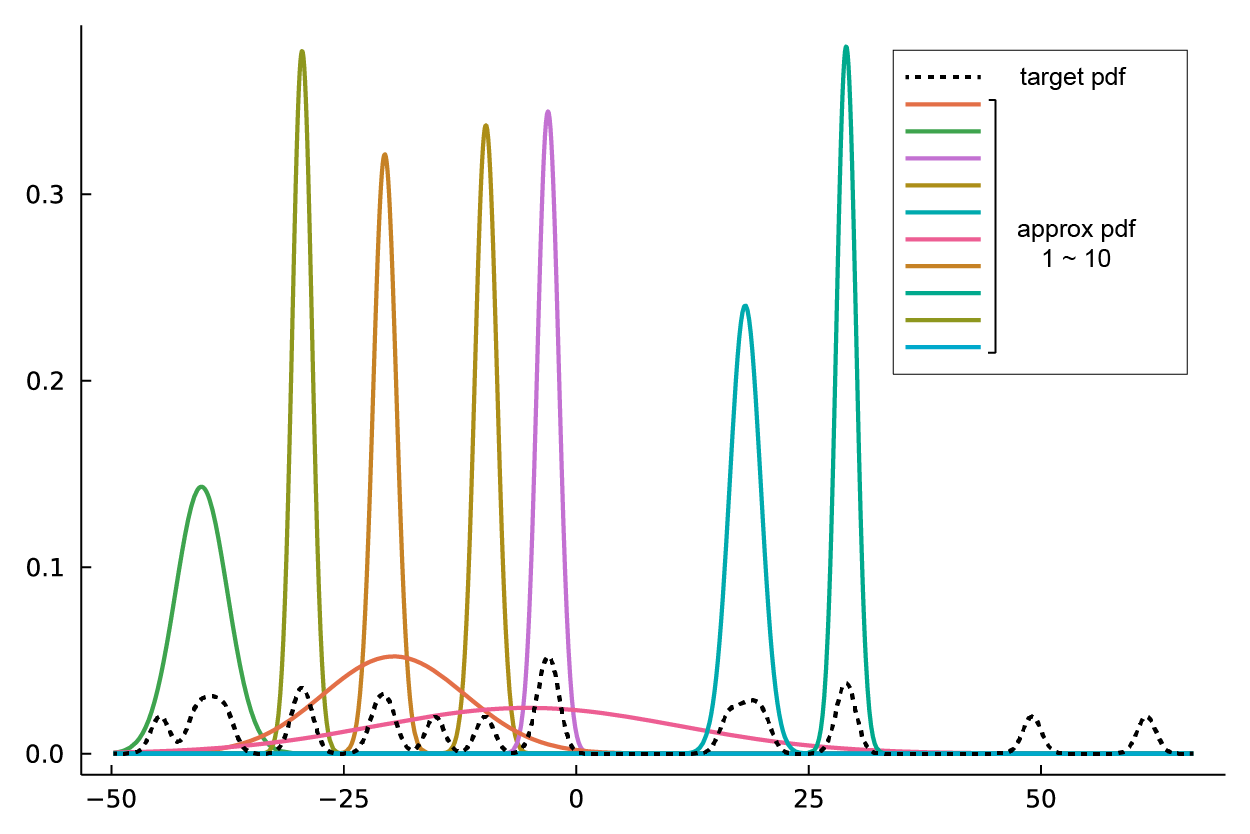}
     \caption{The figure shows the target probability density function (pdf) of the mixture of Gaussians and the pdf curves of the multiple Gaussian approximations obtained by solving the deflated variational inference problem 10 times from the same initial Gaussian solution $\mathcal{N}(\mu = 0, \sigma = \exp(5.0))$.}
     \label{fig:vi}
\end{figure}

\subsection{Pathfinder algorithm for variational inference}\label{sec:examples_pathfinder}

The pathfinder algorithm \citep{pathfinder2021} is a recently proposed algorithm for approximate Bayesian inference. The pathfinder algorithm relies on multi-start optimization of the log joint probability, reusing the trajectory of intermediate solutions and gradients from the optimization to construct a Gaussian approximation per trajectory. 
Re-starting the optimization from a different initial point can result in a different optimal solution and trajectory, leading to a potentially different local Gaussian approximation. 
These Gaussian approximations are then combined and weighted in a mixture of Gaussians which is used to approximate the posterior distribution. For more details, please refer to \citet{pathfinder2021}. 
For simplicity, a target mixture of Gaussians distribution was used in the experiments instead of an unknown posterior distribution. The Pathfinder.jl \footnote{\url{https://github.com/sethaxen/Pathfinder.jl}} package was adapted to use IPOPT \citep{wachter2006implementation} as an optimizer as wrapped in the Nonconvex.jl \footnote{\url{https://github.com/JuliaNonconvex/Nonconvex.jl}} package. 
Instead of random restarts, the same initial solution of 0.0 was used in all the deflation-based pathfinder sub-problems. The target (blue) and approximate (orange) mixtures of Gaussians are shown in figure \ref{fig:pathfinder}. 
The results are almost perfect for a target mixture of 2, 4 and 6 Gaussians. Given that all the optimization sub-problems were started from the same initial solution, this is a fairly positive result. 
For the target mixture of 8 Gaussians, deflation on its own (without random restarts) fails to accurately approximate the target distribution, however this is not entirely surprising.

\begin{figure}[htb]
    \centering
    \includegraphics[width=0.9\columnwidth]{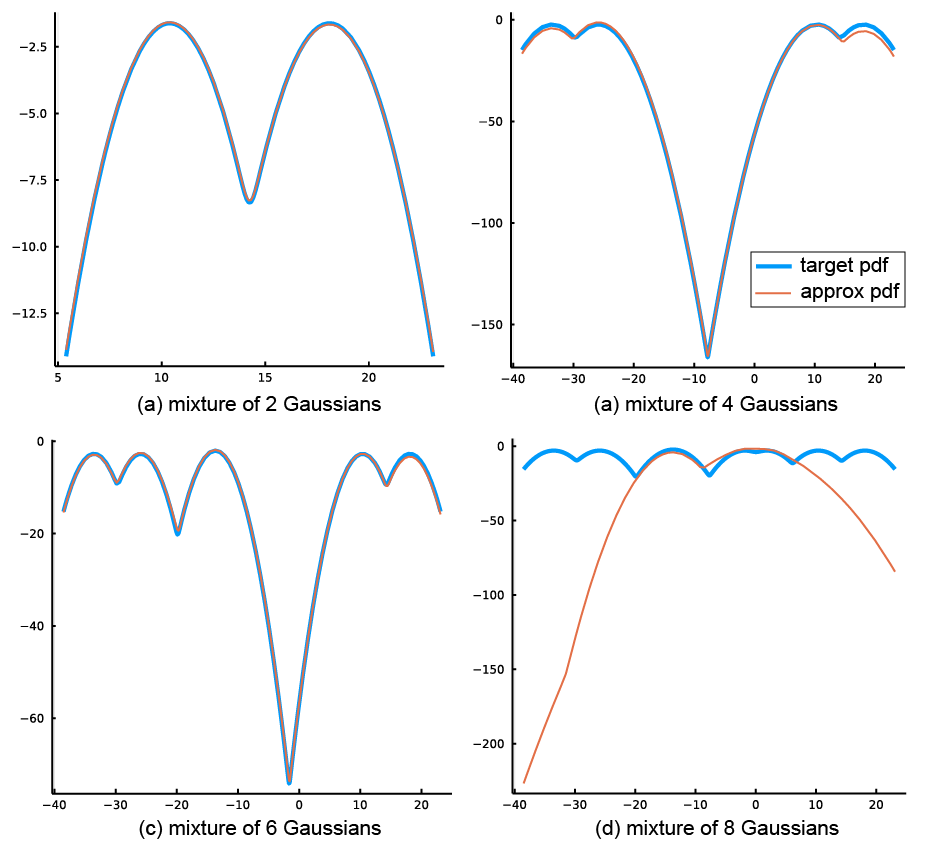}
    \caption{Target (blue) and approximate (red) mixtures of Gaussian probability density functions using the same initial solution of 0.0 in all of the deflation-based sub-problems in the pathfinder algorithm.}
    \label{fig:pathfinder}
\end{figure}

\subsection{Topology optimization - volume constrained compliance minimization}\label{sec:examples_topopt}

The volume-constrained compliance minimization problem from topology optimization (TO) seeks to find designs for physical structures that are as stiff as possible (i.e. least compliant) with respect to known boundary conditions and loading forces while adhering to a given material demand.
It has been well studied and widely applied in mechanical, aerospace, and architectural engineering~\citep{bendsoe2003topology}.
However, because TO problems are high-dimensional, partial differential equation (PDE) constrained, and non-convex, finding multiple local minima that are visually different has been challenging and rarely studied, despite its practical values. 
A compliance-minimizing, volume-constrained TO problem can be formulated as:
\begin{mini}|l|[3]
  {\bm{x}\in\RR^n}{\bm{f} \cdot \bm{u}}{}{}
  \addConstraint {\bm{K}(\bm{x}) \bm{u}(\bm{x})}{= \bm{f}}{\label{eq:topopt}}
  \addConstraint {\frac{\sum_{i=1}^{n} x_i}{n} \leq \tilde{V}}{}{}
  \addConstraint \bm{0} \leq {\bm{x} \leq \bm{1}}{}{}
\end{mini}
where $\bm{x}$ is the pseudo-density of cells in the domain, $\bm{f}$ the given load vector, $\bm{K}(\bm{x})$ the stiffness matrix from the finite-element discretization, $\tilde{V}$ is the user-specified volume fraction.
When $\bx_i = 0$, the corresponding cell is removed and when $\bx_i = 1$ the cell is kept.

We apply our deflation technique by adding an extra constraint to problem \ref{eq:topopt}, using the distance measure that encourages visually different designs mentioned in \cref{sec:distance_measures}:
\begin{align}
   \sum_{k=1}^{\tilde{K}} m(\bx; \bx_k) = \sum_{k=1}^{\tilde{K}} \max(\|\bx - \bx_k\|-r, 0)^{-p} \leq y
\end{align}
where $\tilde{K}$ is the number of found local minima from previous deflation iterations.
In our experiments, $p = 4,\, r = 20,\, y \in (0, 100)$.

Two classic topology optimization benchmark problems are tested here: a continuum MBB beam domain of dimension $120 \times 40$ and a cantilevering truss domain of dimension $40 \times 10$ (\cref{fig:topopt_spec}).
TopOpt.jl\footnote{\url{https://github.com/JuliaTopOpt/TopOpt.jl}} is used for the implementation of problem modeling, finite-element analysis, and density filters. 
We use the widely used Method of Moving Asymptotes (MMA) algorithm \citep{Svanberg1987}
to solve both the undeflated and deflated problems.
Runtime results are presented in \cref{tab:topopt_runtime_results}. 
\cref{fig:cont_topopt} and \cref{fig:truss_topopt} visualize the results of running 20 deflation iterations on both domains. 
The results show that the proposed deflation formulation generates visually different, near-optimal designs deterministically, all starting from the same initial guess.
The runtime results suggest no significant change in the running time due to adding a deflation constraint to an already constrained optimization problem.
\begin{table}[thb]
  \caption{Deflated topology optimization runtime. $^*$: deflated problem runtime is averaged over all the 20 deflation iteration.}
  \begin{center}
    \begin{small}
        \begin{tabular}{lccc}
          \toprule
          TopOpt problems &$\bm{x}$ dim & undeflated & deflated$^*$ \\
          \midrule
          Continuum (\cref{fig:cont_topopt}) & 4800 & 121s & 122s \\
          Truss (\cref{fig:truss_topopt}) & 3608  & 1.9s & 1.6s \\
          \bottomrule
        \end{tabular}
    \end{small}
  \end{center}
  \label{tab:topopt_runtime_results}
\end{table}
\begin{figure}[htb]
     \centering
     \includegraphics[width=0.9\columnwidth]{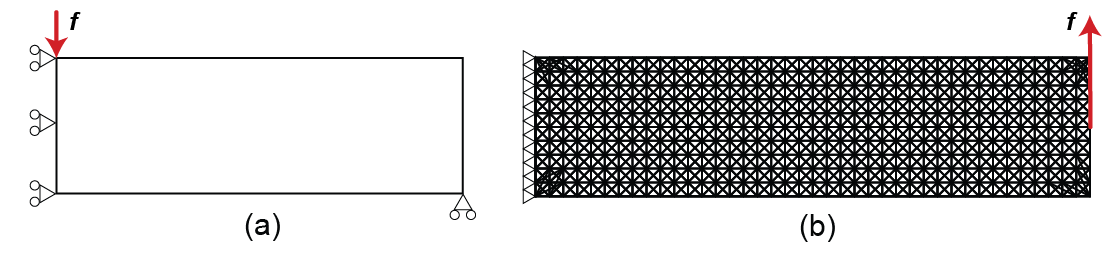}
     \caption{Topology optimization domains: (a) half MBB beam continuum domain; (b) cantilevering truss domain.}
     \label{fig:topopt_spec}
\end{figure}
\begin{figure}[htb]
     \centering
     \includegraphics[width=0.9\columnwidth]{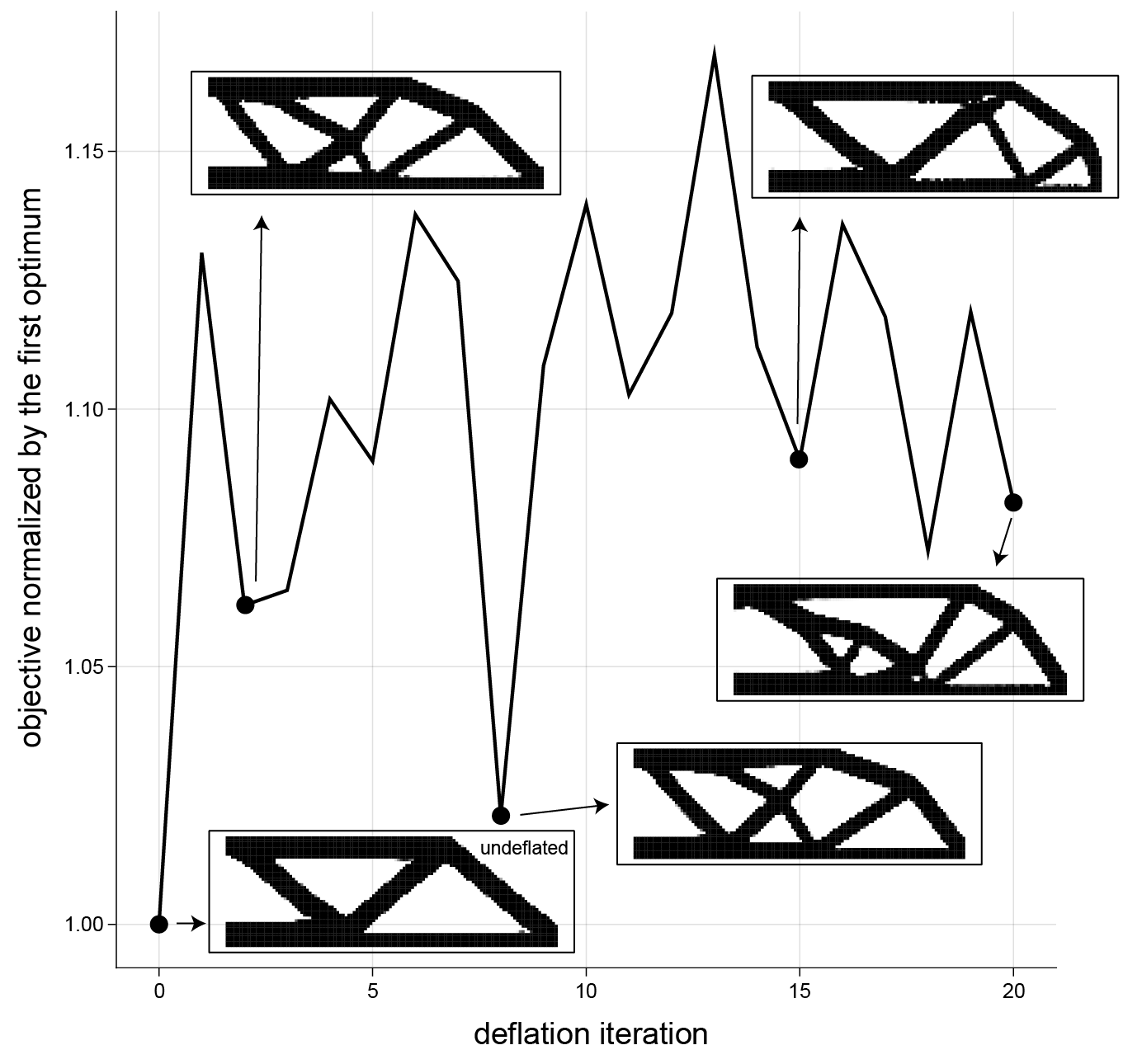}
     \caption{Deflated results of the continuum half MBB beam problem. Iteration 0 is the solution found by solving the undeflated problem.}
     \label{fig:cont_topopt}
\end{figure}
\begin{figure}[]
     \centering
     \includegraphics[width=0.9\columnwidth]{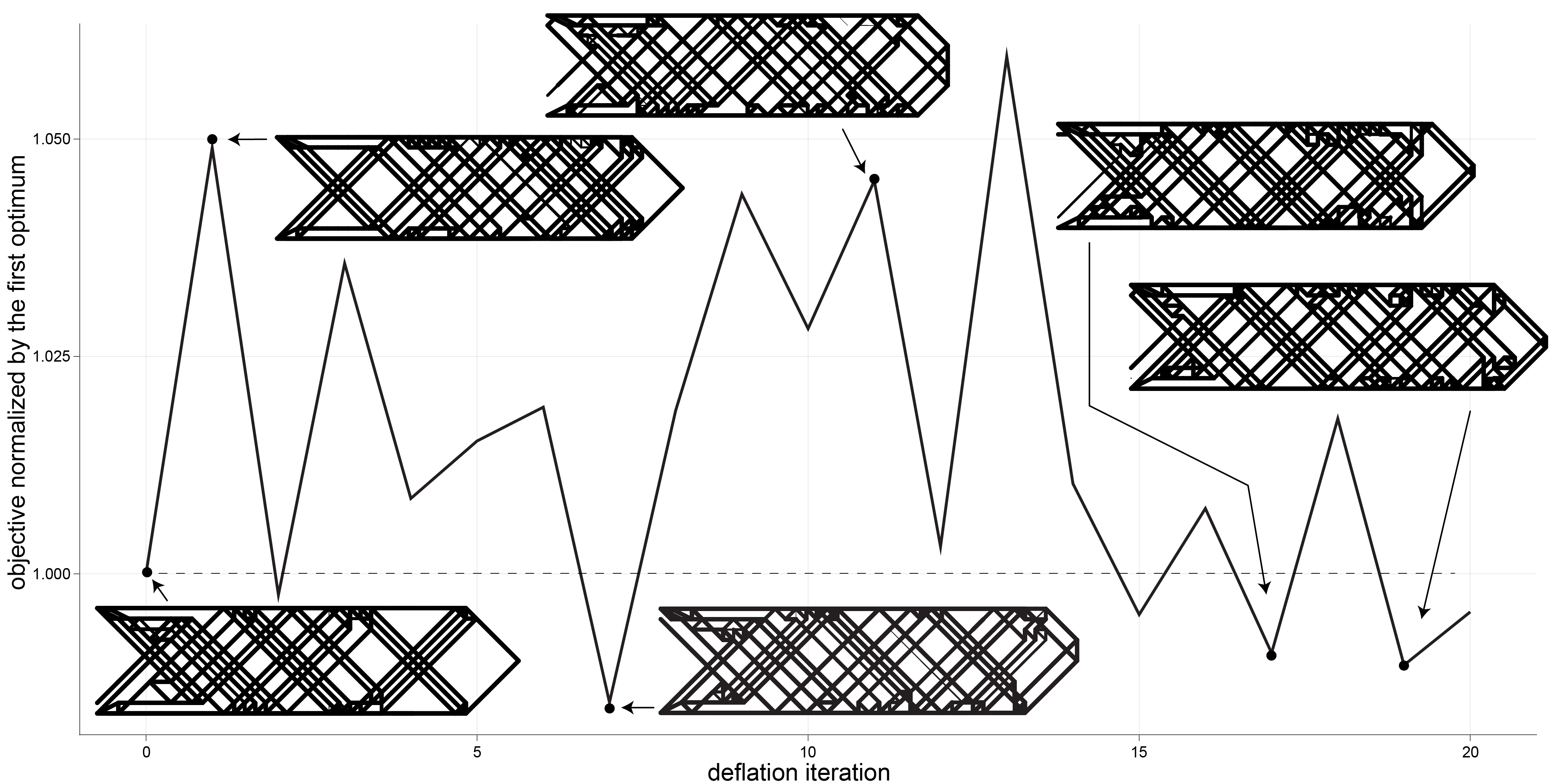}
     \caption{Deflated results of the discrete, cantilevering truss problem. Iteration 0 is the solution found by solving the undeflated problem.}
     \label{fig:truss_topopt}
\end{figure}
\section{Conclusion and future work} \label{sec:conclusion}

In this work, a new way to use deflation was proposed to find diverse solutions of non-convex optimization problems.
With the proposed problem re-formulation, the deflation technique can be easily applied to any non-convex problems using existing, off-the-shelf optimizers.
Promising results of applying the proposed technique on solving problems from topology optimization and variational inference are presented.
We hope our work can enable applications of this simple, yet powerful idea of deflation to a broader set of problems where multiple local optimal solutions are required.
In the future, we hope to make use of deflation-based optimization to enhance optimization-based machine learning algorithms such as maximum likelihood estimation, and to use it for model-based design of experiments in clinical trials to explore multiple possible designs.

\newpage

\bibliographystyle{abbrvnat}
\bibliography{refs}

\newpage

\appendix
\input{appendix}
\end{document}

%% file: preamble.tex
\usepackage{optidef}
\usepackage{bm}
\usepackage{caption}
\usepackage{subcaption}
\usepackage{tikz}

\newcommand{\RR}{\mathbb{R}}
\newcommand{\bx}{\bm{x}}

\usepackage{xcolor}
\definecolor{junglegreen}{rgb}{0.16, 0.67, 0.53}

%% file: appendix.tex
\section{Deflated KKT system of NLP's barrier subproblems}\label{sec:appendix_kkt_logbarrier}
In this section, derivations for the deflated KKT system of the barrier sub-problem used by the interior point method \citep{primaldual} as implemented in the IPOPT software \citep{wachter2006implementation} are presented.

\subsection{KKT system of the barrier problem}

In IPOPT, a log barrier method is used to guarantee that $\bm{l} \leq \bm{x} \leq \bm{u}$ remains satisfied at every intermediate solution if the initial solution is within the bounds. 
The barrier function is defined as:
\begin{align*}
  \mathcal{B}_{\mu}(\bm{x}, \bm{l}, \bm{u}) = -\mu \Big( \sum_i \log{(x_i - l_i)} + \sum_i \log{(u_i - x_i)} \Big)
\end{align*}
for some $\mu > 0$ which would go to $\infty$ if any of the decision variables approaches one of its finite bounds. 
This creates a barrier stopping the optimizer from every reaching the finite bound.
The barrier sub-problem is defined as:
\begin{mini*} 
  {\bm{x}}{\phi_{\mu}(\bm{x}) = f(\bm{x})  + \mathcal{B}_{\mu}(\bm{x}, \bm{l}, \bm{u})}{}{}
\addConstraint {\bm{c}(\bm{x})}{= \bm{0}}{}
\end{mini*}
The KKT stationarity condition is therefore:
\begin{align*}
  \nabla f(\bm{x}) + \nabla \bm{c}(\bm{x})^T \bm{\lambda} - \bm{z}_{\bm{l}} - \bm{z}_{\bm{u}} = \bm{0}
\end{align*}
where $\bm{\lambda}$ is the vector Lagrangian multipliers associated with the equality constraint $\bm{c}(\bm{x}) = \bm{0}$, $\bm{z}_{\bm{l}}$ is a vector whose $i^{th}$ element is
  $z_{l_i} = \frac{\mu}{x_i - l_i}$
and $\bm{z}_{\bm{u}}$ is a vector whose $i^{th}$ element is
  $z_{u_i} = \frac{\mu}{u_i - x_i}$.

Additionally, let $\bm{Z}_{\bm{l}}$ be the diagonal matrix whose diagonal is $\bm{z}_{\bm{l}}$, $\bm{Z}_{\bm{u}}$ be the diagonal matrix whose diagonal is $\bm{z}_{\bm{u}}$, $\bm{X}_{\bm{l}}$ be the diagonal matrix whose diagonal is:
 $\tilde{x}_{l_i} = x_i - l_i$
and $\bm{X}_{\bm{u}}$ be the diagonal matrix whose diagonal is:
 $\tilde{x}_{u_i} = u_i - x_i$.

The first-order KKT sufficient conditions for optimality of the barrier problem, assuming the constraint qualifications are satisfied, can be written as:
\begin{subequations}
\begin{align}
  \nabla f(\bm{x}) + \nabla \bm{c}(\bm{x})^T \bm{\lambda} - \bm{z}_{\bm{l}} - \bm{z}_{\bm{u}} = \bm{0} \\
  \bm{c}(\bm{x}) = \bm{0} \\
  \bm{X}_{\bm{l}} \bm{Z}_{\bm{l}} \bm{1} - \mu \bm{1} = \bm{0} \label{eqn:reciprocal1} \\
  \bm{X}_{\bm{u}} \bm{Z}_{\bm{u}} \bm{1} - \mu \bm{1} = \bm{0} \label{eqn:reciprocal2} \\
  \bm{l} \leq \bm{x} \leq \bm{u} \\
  \bm{z}_{\bm{l}} \geq \bm{0} \\
  \bm{z}_{\bm{u}} \geq \bm{0}
\end{align}
\end{subequations}
where $\bm{1}$ is a vector of ones. 
Conditions \ref{eqn:reciprocal1} and \ref{eqn:reciprocal2} ensure that the relationship between $\bm{X}_{\bm{l}}$, $\bm{Z}_{\bm{l}}$, $\bm{X}_{\bm{u}}$ and $\bm{Z}_{\bm{u}}$ is maintained according to the definitions of $\bm{z}_{\bm{l}}$ and $\bm{z}_{\bm{u}}$. 
In an interior point algorithm (e.g. IPOPT~\cite{wachter2006implementation}), primal-dual solutions to the equality KKT conditions are found using a Newton-like method while ensuing that the inequality conditions are satisfied by projection.

In order to solve the barrier problem for a given value $\mu = \mu_j$, a damped Newton's method is usually applied to the primal-dual optimality conditions. 
Here we use $k$ to denote the iteration counter for the inner iterations when solving the barrier problem. Given an iterate $(\bm{x}_k, \bm{\lambda}_k, \bm{z}_{\bm{l}, k}, \bm{z}_{\bm{u}, k})$ with $\bm{l} < \bm{x}_k < \bm{u}$ and  $\bm{z}_{\bm{l}, k}, \bm{z}_{\bm{u}, k} > \bm{0}$, some search directions $(\bm{d}_k^{\bm{x}}, \bm{d}_k^{\bm{\lambda}}, \bm{d}_k^{\bm{z}_{\bm{l}}}, \bm{d}_k^{\bm{z}_{\bm{u}}})$ are obtained using the regularized linearization of the optimality conditions (excluding the inequality conditions):
\begin{align}
  \begin{bmatrix}
    \bm{W}_k & \bm{A}_k & -\bm{I} & -\bm{I} \\
    \bm{A}_k^T & \bm{0} & \bm{0} & \bm{0} \\
    \bm{Z}_{\bm{l},k} & \bm{0} & \bm{X}_{\bm{l},k} & \bm{0} \\
    \bm{Z}_{\bm{u},k} & \bm{0} & \bm{0} & \bm{X}_{\bm{u}, k}
  \end{bmatrix} \begin{pmatrix}
    \bm{d}_k^{\bm{x}} \\
    \bm{d}_k^{\bm{\lambda}} \\
    \bm{d}_k^{\bm{z}_{\bm{l}}} \\
    \bm{d}_k^{\bm{z}_{\bm{u}}}
  \end{pmatrix} = 
    - \begin{pmatrix}
    \nabla f(\bm{x}_k) + \bm{A}_k^T \bm{\lambda}_k - \bm{z}_{\bm{l}, k} - \bm{z}_{\bm{u}, k}  \\
    \bm{c}(\bm{x}_k) \\
    \bm{X}_{\bm{l}, k} \bm{Z}_{\bm{l}, k} \bm{1} - \mu_j \bm{1} \\
    \bm{X}_{\bm{u}, k} \bm{Z}_{\bm{u}, k} \bm{1} - \mu_j \bm{1}
  \end{pmatrix}
\end{align}

where $\bm{A}_k := \nabla \bm{c}(\bm{x}_k)^T$ and $\bm{W}_k := \nabla^2_{\bm{x}\bm{x}} (f(\bx_k) + \bm{c}(\bm{x}_k)^T \bm{\lambda}_k)$ is the Hessian of the Lagrangian function of the original problem.
The Lagrangian terms from the bounds constraints are ignored because they don't contribute to the Hessian. When the Hessian of the Lagrangian is not available, a l-BFGS approximation \citep{nocedal2006numerical} of the Hessian can be used instead. This changes the IPOPT algorithm from a second order algorithm to a first order one. And the Newton update becomes a quasi-Newton update.

Instead of solving the non-symmetric system of equations above, one can instead change the system as such:
\begin{align}
  \begin{bmatrix}
    \bm{W}_k & \bm{A}_k & -\bm{I} & -\bm{I} \\
    \bm{A}_k^T & \bm{0} & \bm{0} & \bm{0} \\
    \bm{X}_{\bm{l}}^{-1} \bm{Z}_{\bm{l},k} & \bm{0} & \bm{I} & \bm{0} \\
    \bm{X}_{\bm{u}}^{-1} \bm{Z}_{\bm{u},k} & \bm{0} & \bm{0} & \bm{I}
  \end{bmatrix} \begin{pmatrix}
    \bm{d}_k^{\bm{x}} \\
    \bm{d}_k^{\bm{\lambda}} \\
    \bm{d}_k^{\bm{z}_{\bm{l}}} \\
    \bm{d}_k^{\bm{z}_{\bm{u}}}
  \end{pmatrix} = 
  - \begin{pmatrix}
    \nabla f(\bm{x}_k) + \bm{A}_k^T \bm{\lambda}_k - \bm{z}_{\bm{l}, k} - \bm{z}_{\bm{u}, k}  \\
    \bm{c}(\bm{x}_k) \\
    \bm{z}_{\bm{l}, k} - \mu_j \bm{X}_{\bm{l}}^{-1} \bm{1} \\
    \bm{z}_{\bm{u}, k} - \mu_j \bm{X}_{\bm{u}}^{-1} \bm{1}
  \end{pmatrix}
\end{align}
Adding the third and fourth equations to the first one, the third and fourth blocks of coefficients of the first equation will be eliminated.
\begin{align}
  \begin{bmatrix}
    \bm{W}_k + \bm{\Sigma}_k & \bm{A}_k & \bm{0} & \bm{0} \\
    \bm{A}_k^T & \bm{0} & \bm{0} & \bm{0} \\
    \bm{X}_{\bm{l}}^{-1} \bm{Z}_{\bm{l},k} & \bm{0} & \bm{I} & \bm{0} \\
    \bm{X}_{\bm{u}}^{-1} \bm{Z}_{\bm{u},k} & \bm{0} & \bm{0} & \bm{I}
  \end{bmatrix} \begin{pmatrix}
    \bm{d}_k^{\bm{x}} \\
    \bm{d}_k^{\bm{\lambda}} \\
    \bm{d}_k^{\bm{z}_{\bm{l}}} \\
    \bm{d}_k^{\bm{z}_{\bm{u}}}
  \end{pmatrix} = 
    -\begin{pmatrix}
    \nabla f(\bm{x}_k) + \bm{A}_k^T \bm{\lambda}_k - \mu_j \bm{X}_{\bm{l}}^{-1} \bm{1} - \mu_j \bm{X}_{\bm{u}}^{-1} \bm{1}  \\
    \bm{c}(\bm{x}_k) \\
    \bm{z}_{\bm{l}, k} - \mu_j \bm{X}_{\bm{l}}^{-1} \bm{1} \\
    \bm{z}_{\bm{u}, k} - \mu_j \bm{X}_{\bm{u}}^{-1} \bm{1}
  \end{pmatrix}
\end{align}
where $\bm{\Sigma}_k = \bm{X}_{\bm{l}}^{-1} \bm{Z}_{\bm{l},k} + \bm{X}_{\bm{u}}^{-1} \bm{Z}_{\bm{u},k}$. Therefore, one can now solve for $\bm{d}_k^{\bm{x}}$ and $\bm{d}_k^{\bm{\lambda}}$ by solving the following symmetric linear system:
\begin{align}
  \begin{bmatrix}
    \bm{W}_k + \bm{\Sigma}_k & \bm{A}_k \\
    \bm{A}_k^T & \bm{0}
  \end{bmatrix} \begin{pmatrix}
    \bm{d}_k^{\bm{x}} \\
    \bm{d}_k^{\bm{\lambda}}
  \end{pmatrix} = 
  - \begin{pmatrix}
    \nabla f(\bm{x}_k) + \bm{A}_k^T \bm{\lambda}_k - \mu_j \bm{X}_{\bm{l}}^{-1} \bm{1} - \mu_j \bm{X}_{\bm{u}}^{-1} \bm{1}  \\
    \bm{c}(\bm{x}_k)
  \end{pmatrix}
  \label{eq:ipopt_linearized_kkt}
\end{align}
then use the value of $\bm{d}_k^{\bm{x}}$ to find $\bm{d}_k^{\bm{z}_{\bm{l}}}$ and $\bm{d}_k^{\bm{z}_{\bm{u}}}$ using:
\begin{subequations}
\begin{align}
  \bm{d}_k^{\bm{z}_{\bm{l}}} = - \bm{z}_{\bm{l}, k} + \mu_j \bm{X}_{\bm{l}}^{-1} \bm{1} - \bm{X}_{\bm{l}}^{-1} \bm{Z}_{\bm{l}, k} \bm{d}_k^{\bm{x}} \\
  \bm{d}_k^{\bm{z}_{\bm{u}}} = - \bm{z}_{\bm{u}, k} + \mu_j \bm{X}_{\bm{u}}^{-1} \bm{1} - \bm{X}_{\bm{u}}^{-1} \bm{Z}_{\bm{u}, k} \bm{d}_k^{\bm{x}}
\end{align}
\end{subequations}

\subsection{Deflating the KKT system} \label{sec:appendix_deflated_ipopt}


If we define the RHS of \cref{eq:ipopt_linearized_kkt} as:
\begin{align}
    \bm{F}(\bm{x}, \bm{\lambda}) = \begin{pmatrix}
    \nabla f(\bm{x}) + \bm{A}^T \bm{\lambda} - \mu_j \bm{X}_{\bm{l}}^{-1} \bm{1} - \mu_j \bm{X}_{\bm{u}}^{-1} \bm{1}  \\
    \bm{c}(\bm{x})
  \end{pmatrix}
\end{align}
Then the goal of the primal-dual optimizer then becomes solving for $\bm{F}(\bm{x}, \bm{\lambda}) = \bm{0}$. 

If we apply the deflation operator to the entire $\bm{F}(\bm{x}, \bm{\lambda})$, we will obtain $G(\bm{x}, \bm{\lambda}) = \bm{M}(\bm{x}; \bm{x}_1) \bm{F}(\bm{x}, \bm{\lambda}) = 0$, whose Jacobian is:
%
\begin{align*}
    \nabla_{\bm{x}} \bm{G}(\bm{x}) = m(\bm{x}; \bm{x}_1) \nabla_{\bm{x}} \bm{F}(\bm{x}, \bm{\lambda}) + 
    diag(F(\bm{x}, \bm{\lambda})) \bm{1} (\nabla_{\bm{x}} m(\bm{x}; \bm{x}_1))^T
\end{align*}
where $\nabla_{\bm{x}} m(\bm{x}; \bm{x}_1)$ is the gradient vector of the deflation function, $\bm{1}$ is a column vector of 1s and $diag(\bm{F}(\bm{x}))$ is a diagonal matrix with diagonal $\bm{F}(\bm{x})$.
The Jacobian of this function wrt $\bm{x}$ is:
\begin{align}
  \begin{bmatrix}
    \bm{W} + \bm{\Sigma} \\
    \bm{A}^T
  \end{bmatrix}
\end{align}
where $\bm{W}$, $\bm{\Sigma}$ and $\bm{A}$ are defined as $\bm{W}_k$, $\bm{\Sigma}_k$ and $\bm{A}_k$ in the last section.
The top matrix is symmetric because it is the Hessian matrix of the barrier objective function. 
However, $\bm{A}_k := \nabla \bm{c}(\bm{x}_k)^T$ in general is not symmetric, and thus prevents the usage of efficient linear algebra algorithms for symmetric linear system.

%% file: arxiv_version.bbl
\begin{thebibliography}{43}
\providecommand{\natexlab}[1]{#1}
\providecommand{\url}[1]{\texttt{#1}}
\expandafter\ifx\csname urlstyle\endcsname\relax
  \providecommand{\doi}[1]{doi: #1}\else
  \providecommand{\doi}{doi: \begingroup \urlstyle{rm}\Url}\fi

\bibitem[Arnoud et~al.(2019)Arnoud, Guvenen, and
  Kleineberg]{arnoud_benchmarking_2019}
A.~Arnoud, F.~Guvenen, and T.~Kleineberg.
\newblock Benchmarking {Global} {Optimizers}.
\newblock Technical Report w26340, National Bureau of Economic Research,
  Cambridge, MA, Oct. 2019.
\newblock URL \url{http://www.nber.org/papers/w26340.pdf}.

\bibitem[Barron and Gomez(1991)]{tunneling3}
C.~Barron and S.~Gomez.
\newblock \emph{The exponential tunneling method}.
\newblock 1991.

\bibitem[Belotti et~al.(2009)Belotti, Lee, Liberti, Margot, and
  W{\"a}chter]{BeLeLiMaWa08}
P.~Belotti, J.~Lee, L.~Liberti, F.~Margot, and A.~W{\"a}chter.
\newblock Branching and bounds tightening techniques for non-convex {MINLP}.
\newblock \emph{Optimization Methods and Software}, 24\penalty0 (4-5):\penalty0
  597--634, 2009.

\bibitem[Bendsoe and Sigmund(2003)]{bendsoe2003topology}
M.~P. Bendsoe and O.~Sigmund.
\newblock \emph{Topology optimization: theory, methods, and applications}.
\newblock Springer Science \& Business Media, 2003.

\bibitem[Bertsekas(1996)]{Bertsekas1996}
D.~P. Bertsekas.
\newblock \emph{Constrained Optimization and Lagrange Multiplier Methods}.
\newblock Athena Scientific, 1996.

\bibitem[Brow and Gearhart(1971)]{brow1971deflation}
K.~M. Brow and W.~B. Gearhart.
\newblock Deflation techniques for the calculation of further solutions of a
  nonlinear system.
\newblock \emph{Numerische Mathematik}, 16\penalty0 (4):\penalty0 334--342,
  1971.

\bibitem[Duchi et~al.(2011)Duchi, Hazan, and Singer]{adagrad2011}
J.~Duchi, E.~Hazan, and Y.~Singer.
\newblock Adaptive subgradient methods for online learning and stochastic
  optimization.
\newblock \emph{Journal of Machine Learning Research}, 12\penalty0
  (61):\penalty0 2121--2159, 2011.
\newblock URL \url{http://jmlr.org/papers/v12/duchi11a.html}.

\bibitem[Falkner et~al.(2018)Falkner, Klein, and Hutter]{falkner2018bohb}
S.~Falkner, A.~Klein, and F.~Hutter.
\newblock Bohb: Robust and efficient hyperparameter optimization at scale,
  2018.

\bibitem[Farrell et~al.(2016)Farrell, Beentjes, and
  Birkisson]{farrell_computation_2016}
P.~E. Farrell, C.~H.~L. Beentjes, and A.~Birkisson.
\newblock The computation of disconnected bifurcation diagrams.
\newblock \emph{arXiv:1603.00809 [math]}, Mar. 2016.
\newblock URL \url{http://arxiv.org/abs/1603.00809}.
\newblock arXiv: 1603.00809.

\bibitem[Farrell et~al.(2020)Farrell, Croci, and
  Surowiec]{farrell2020deflation}
P.~E. Farrell, M.~Croci, and T.~M. Surowiec.
\newblock Deflation for semismooth equations.
\newblock \emph{Optimization Methods and Software}, 35\penalty0 (6):\penalty0
  1248--1271, 2020.

\bibitem[Fiacco and Mccormick(1968)]{logbarrier}
A.~Fiacco and G.~Mccormick.
\newblock \emph{Nonlinear Programming: Sequential Unconstrained Minimization
  Techniques.}
\newblock John Wiley, New York, 1968.

\bibitem[Gablonsky and Kelley(2001)]{gablonsky_locally-biased}
J.~M. Gablonsky and C.~T. Kelley.
\newblock A {Locally}-{Biased} form of the {DIRECT} {Algorithm}.
\newblock \emph{Journal of Global Optimization}, 21:\penalty0 27--37, 2001.

\bibitem[{Gecode Team}()]{gecode}
{Gecode Team}.
\newblock Gecode: Generic constraint development environment.
\newblock Available from \texttt{http://www.gecode.org}.

\bibitem[Gendreau and Potvin(2010)]{GendPotv10}
M.~Gendreau and J.-Y. Potvin, editors.
\newblock \emph{Handbook of metaheuristics}.
\newblock Springer, New York, NY, USA, 2 edition, 2010.

\bibitem[Gomez and Levy(1982)]{tunneling2}
S.~Gomez and A.~V. Levy.
\newblock The tunnelling method for solving the constrained global optimization
  problem with several non-connected feasible regions.
\newblock \emph{Numerical analysis}, page 34–47, 1982.

\bibitem[Gomez et~al.(2003)Gomez, del Castillo, Castellanos, and
  Solano]{tunneling4}
S.~Gomez, N.~del Castillo, L.~Castellanos, and J.~Solano.
\newblock The parallel tunneling method.
\newblock \emph{Parallel Computing}, page 523–533, 2003.

\bibitem[Gudmundsson(1998)]{stogo1}
S.~Gudmundsson.
\newblock \emph{Parallel Global Optimization, M.Sc. Thesis, IMM, Technical
  University of Denmark}, 1998.

\bibitem[Jones et~al.(1993)Jones, Perttunen, and
  Stuckman]{jones_lipschitzian_1993}
D.~R. Jones, C.~D. Perttunen, and B.~E. Stuckman.
\newblock Lipschitzian optimization without the {Lipschitz} constant.
\newblock \emph{Journal of Optimization Theory and Applications}, 79\penalty0
  (1):\penalty0 157--181, Oct. 1993.
\newblock ISSN 0022-3239, 1573-2878.
\newblock \doi{10.1007/BF00941892}.
\newblock URL \url{http://link.springer.com/10.1007/BF00941892}.

\bibitem[K.~Madsen and Zilinskas(1998)]{stogo2}
S.~Z. K.~Madsen and A.~Zilinskas.
\newblock \emph{Global Optimization using Branch-and-Bound}, 1998.

\bibitem[Kaelo and Ali(2006)]{kaelo_variants_2006}
P.~Kaelo and M.~M. Ali.
\newblock Some {Variants} of the {Controlled} {Random} {Search} {Algorithm} for
  {Global} {Optimization}.
\newblock \emph{Journal of Optimization Theory and Applications}, 130\penalty0
  (2):\penalty0 253--264, Dec. 2006.
\newblock ISSN 0022-3239, 1573-2878.
\newblock \doi{10.1007/s10957-006-9101-0}.
\newblock URL \url{http://link.springer.com/10.1007/s10957-006-9101-0}.

\bibitem[Kucherenko and Sytsko(2005)]{kucherenko_application_2005}
S.~Kucherenko and Y.~Sytsko.
\newblock Application of {Deterministic} {Low}-{Discrepancy} {Sequences} in
  {Global} {Optimization}.
\newblock \emph{Computational Optimization and Applications}, 30\penalty0
  (3):\penalty0 297--318, Mar. 2005.
\newblock ISSN 0926-6003, 1573-2894.
\newblock \doi{10.1007/s10589-005-4615-1}.
\newblock URL \url{http://link.springer.com/10.1007/s10589-005-4615-1}.

\bibitem[Kucukelbir et~al.(2015)Kucukelbir, Ranganath, Gelman, and
  Blei]{advi2015}
A.~Kucukelbir, R.~Ranganath, A.~Gelman, and D.~Blei.
\newblock Automatic variational inference in stan.
\newblock In C.~Cortes, N.~Lawrence, D.~Lee, M.~Sugiyama, and R.~Garnett,
  editors, \emph{Advances in Neural Information Processing Systems}, volume~28.
  Curran Associates, Inc., 2015.
\newblock URL
  \url{https://proceedings.neurips.cc/paper/2015/file/352fe25daf686bdb4edca223c921acea-Paper.pdf}.

\bibitem[Levy and Gomez(1981)]{tunneling1}
A.~V. Levy and S.~Gomez.
\newblock The tunneling method applied to global optimization.
\newblock \emph{Numerical optimization}, page 213–244, 1981.

\bibitem[Li et~al.(2018)Li, Jamieson, DeSalvo, Rostamizadeh, and
  Talwalkar]{li2018hyperband}
L.~Li, K.~Jamieson, G.~DeSalvo, A.~Rostamizadeh, and A.~Talwalkar.
\newblock Hyperband: A novel bandit-based approach to hyperparameter
  optimization, 2018.

\bibitem[McKay M.~D.(1979)]{latin_hypercube}
C.~W.~J. McKay M.~D., Beckman R.~J.
\newblock A comparison of three methods for selecting values of input variables
  in the analysis of output from a computer code.
\newblock \emph{Technometrics}, 21\penalty0 (2):\penalty0 239--245, 1979.
\newblock \doi{10.2307/1268522}.

\bibitem[Moscato and Cotta(2010)]{Moscato2010}
P.~Moscato and C.~Cotta.
\newblock \emph{A Modern Introduction to Memetic Algorithms}, pages 141--183.
\newblock Springer US, Boston, MA, 2010.
\newblock ISBN 978-1-4419-1665-5.
\newblock \doi{10.1007/978-1-4419-1665-5_6}.
\newblock URL \url{https://doi.org/10.1007/978-1-4419-1665-5_6}.

\bibitem[Nagarajan et~al.(2016)Nagarajan, Lu, Yamangil, and
  Bent]{alpine_CP2016}
H.~Nagarajan, M.~Lu, E.~Yamangil, and R.~Bent.
\newblock Tightening {McCormick} relaxations for nonlinear programs via dynamic
  multivariate partitioning.
\newblock In \emph{International Conference on Principles and Practice of
  Constraint Programming}, pages 369--387. Springer, 2016.
\newblock \doi{10.1007/978-3-319-44953-1_24}.

\bibitem[Nagarajan et~al.(2019)Nagarajan, Lu, Wang, Bent, and
  Sundar]{alpine_JOGO2019}
H.~Nagarajan, M.~Lu, S.~Wang, R.~Bent, and K.~Sundar.
\newblock An adaptive, multivariate partitioning algorithm for global
  optimization of nonconvex programs.
\newblock \emph{Journal of Global Optimization}, 2019.
\newblock ISSN 1573-2916.
\newblock \doi{10.1007/s10898-018-00734-1}.

\bibitem[Nocedal and Wright(2006)]{nocedal2006numerical}
J.~Nocedal and S.~Wright.
\newblock \emph{Numerical optimization}.
\newblock Springer Science \& Business Media, 2006.

\bibitem[Papadopoulos et~al.(2021)Papadopoulos, Farrell, and
  Surowiec]{papadopoulos_computing_2021}
I.~P.~A. Papadopoulos, P.~E. Farrell, and T.~M. Surowiec.
\newblock Computing multiple solutions of topology optimization problems.
\newblock \emph{arXiv:2004.11797 [cs, math]}, Jan. 2021.
\newblock URL \url{http://arxiv.org/abs/2004.11797}.
\newblock arXiv: 2004.11797.

\bibitem[Ratschek and Rokne(2007)]{ratschek}
H.~Ratschek and J.~Rokne.
\newblock \emph{New computer methods for global optimization}, 2007.

\bibitem[Rinnooy~Kan and Timmer(1987)]{rinnooy_kan_stochastic_1987}
A.~H.~G. Rinnooy~Kan and G.~T. Timmer.
\newblock Stochastic global optimization methods part {I}: {Clustering}
  methods.
\newblock \emph{Mathematical Programming}, 39\penalty0 (1):\penalty0 27--56,
  Sept. 1987.
\newblock ISSN 0025-5610, 1436-4646.
\newblock \doi{10.1007/BF02592070}.
\newblock URL \url{http://link.springer.com/10.1007/BF02592070}.

\bibitem[Sahinidis(2017)]{sahinidis:baron:21.1.13}
N.~V. Sahinidis.
\newblock \emph{{BARON 21.1.13: Global Optimization of Mixed-Integer Nonlinear
  Programs, {\em User's Manual}}}, 2017.

\bibitem[Sinha et~al.(2018)Sinha, Malo, and Deb]{bilevel2018}
A.~Sinha, P.~Malo, and K.~Deb.
\newblock A review on bilevel optimization: From classical to evolutionary
  approaches and applications.
\newblock \emph{IEEE Transactions on Evolutionary Computation}, 22\penalty0
  (2):\penalty0 276--295, 2018.
\newblock \doi{10.1109/TEVC.2017.2712906}.

\bibitem[Svanberg(1987)]{Svanberg1987}
K.~Svanberg.
\newblock {The method of moving asymptotes - a new method for structural
  optimization}.
\newblock \emph{International Journal for Numerical Methods in Engineering},
  24\penalty0 (2):\penalty0 359--373, 1987.

\bibitem[Svanberg(2002)]{Svanberg2002}
K.~Svanberg.
\newblock {A Class of Globally Convergent Optimization Methods Based on
  Conservative Convex Separable Approximations}.
\newblock \emph{SIAM Journal on Optimization}, 12\penalty0 (2):\penalty0
  555--573, 2002.

\bibitem[Tawarmalani and Sahinidis(2005)]{tawarmalani_polyhedral_2005}
M.~Tawarmalani and N.~V. Sahinidis.
\newblock A polyhedral branch-and-cut approach to global optimization.
\newblock \emph{Mathematical Programming}, 103\penalty0 (2):\penalty0 225--249,
  June 2005.
\newblock ISSN 0025-5610, 1436-4646.
\newblock \doi{10.1007/s10107-005-0581-8}.
\newblock URL \url{http://link.springer.com/10.1007/s10107-005-0581-8}.

\bibitem[Vigerske and Gleixner(2018)]{SCIP}
S.~Vigerske and A.~Gleixner.
\newblock Scip: global optimization of mixed-integer nonlinear programs in a
  branch-and-cut framework.
\newblock \emph{Optimization Methods and Software}, 33\penalty0 (3):\penalty0
  563--593, 2018.
\newblock \doi{10.1080/10556788.2017.1335312}.
\newblock URL \url{https://doi.org/10.1080/10556788.2017.1335312}.

\bibitem[W{\"a}chter and Biegler(2006)]{wachter2006implementation}
A.~W{\"a}chter and L.~T. Biegler.
\newblock On the implementation of an interior-point filter line-search
  algorithm for large-scale nonlinear programming.
\newblock \emph{Mathematical programming}, 106\penalty0 (1):\penalty0 25--57,
  2006.

\bibitem[Wilhelm and Stuber(2020)]{EAGO}
M.~E. Wilhelm and M.~D. Stuber.
\newblock Eago.jl: easy advanced global optimization in julia.
\newblock \emph{Optimization Methods and Software}, pages 1--26, 2020.
\newblock \doi{10.1080/10556788.2020.1786566}.
\newblock URL \url{https://doi.org/10.1080/10556788.2020.1786566}.

\bibitem[Wright(1997)]{primaldual}
S.~Wright.
\newblock \emph{Primal-dual interior-point methods}.
\newblock 1997.
\newblock \doi{10.1137/1.9781611971453}.
\newblock URL \url{https://epubs.siam.org/doi/abs/10.1137/1.9781611971453}.

\bibitem[Zhang et~al.(2021)Zhang, Carpenter, Gelman, and
  Vehtari]{pathfinder2021}
L.~Zhang, B.~Carpenter, A.~Gelman, and A.~Vehtari.
\newblock Pathfinder: Parallel quasi-newton variational inference, 2021.

\bibitem[Zhang and Norato(2018)]{zhang_finding_2018}
S.~Zhang and J.~A. Norato.
\newblock Finding {Better} {Local} {Optima} in {Topology} {Optimization} via
  {Tunneling}.
\newblock In \emph{Volume {2B}: 44th {Design} {Automation} {Conference}}, page
  V02BT03A014, Quebec City, Quebec, Canada, Aug. 2018. American Society of
  Mechanical Engineers.
\newblock ISBN 978-0-7918-5176-0.
\newblock \doi{10.1115/DETC2018-86116}.
\newblock URL
  \url{https://asmedigitalcollection.asme.org/IDETC-CIE/proceedings/IDETC-CIE2018/51760/Quebec%20City,%20Quebec,%20Canada/274787}.

\end{thebibliography}
